\documentclass{article}
\usepackage{mathrsfs, amsmath, amsthm, amssymb, bm, mathtools, thm-restate}
\usepackage{graphicx, xcolor, tikz}
\usetikzlibrary{arrows.meta, intersections, calc, shapes}
\usepackage{caption, subcaption}
\usepackage{array}
\makeatletter
\def\th@plain{%
  \upshape 
}
\makeatother

\makeatletter
\renewenvironment{proof}[1][\proofname]{\par
  \pushQED{\qed}%
  \normalfont \topsep6\p@\@plus6\p@\relax
  \trivlist
  \item[\hskip\labelsep
        \bfseries
    #1\@addpunct{.}]\ignorespaces
}{%
  \popQED\endtrivlist\@endpefalse
}
\makeatother

\newtheorem{theorem}{Theorem}[section]
\newtheorem{lemma}{Lemma}[section]
\newtheorem{corollary}[theorem]{Corollary}
\theoremstyle{definition}
\newtheorem{definition}{Definition}

\usepackage[left=25mm, top=25mm, bottom=25mm, right=25mm]{geometry}
\usepackage[T1]{fontenc} 
\usepackage{paralist, tablists}
\usepackage[inline]{enumitem}
\usepackage[square, numbers, sort&compress]{natbib}

\usepackage[pdftex,%
  bookmarks=true,%
  bookmarksnumbered=true, 
  bookmarksopen=true, 
  plainpages=false,%
  pdfpagelabels,%
  colorlinks=true, 
  linkcolor=blue, 
  citecolor=blue,%
  anchorcolor=green,
  urlcolor=blue,
  breaklinks=true,
  hyperindex=true]{hyperref}

\newcommand{\etal}{et~al.\ }
\newcommand{\ie}{i.e.,\ }
\DeclareMathOperator {\dist}{dist}
\usepackage{marvosym,wasysym}

\usepackage{colortbl, tabularx, fontawesome}
\usepackage{cancel}
\definecolor{LGrey}{rgb}{.7,.7,.7}
\usepackage{threeparttablex}
\title{DP-3-coloring of planar graphs without certain cycles}
\author{Mengjiao Rao \qquad Tao Wang\footnote{{\tt Corresponding
author: wangtao@henu.edu.cn; iwangtao8@gmail.com} } \\
{\small Institute of Applied Mathematics}\\
{\small Henan University, Kaifeng, 475004, P. R. China}}
\begin{document}
\date{}
\maketitle
\begin{abstract}
DP-coloring is a generalization of list-coloring, which was introduced by Dvo\v{r}\'{a}k and Postle. Zhang showed that every planar graph with neither adjacent triangles nor 5-, 6-, 9-cycles is 3-choosable. Liu \etal showed that every planar graph without 4-, 5-, 6- and 9-cycles is DP-3-colorable. In this paper, we show that every planar graph with neither adjacent triangles nor 5-, 6-, 9-cycles is DP-3-colorable, which generalizes these results. Yu \etal gave three Bordeaux-type results by showing that: (i) every planar graph with the distance of triangles at least three and no 4-, 5-cycles is DP-3-colorable; (ii) every planar graph with the distance of triangles at least two and no 4-, 5-, 6-cycles is DP-3-colorable; (iii) every planar graph with the distance of triangles at least two and no 5-, 6-, 7-cycles is DP-3-colorable. We also give two Bordeaux-type results in the last section: (i) every planar graph with neither 5-, 6-, 8-cycles nor triangles at distance less than two is DP-3-colorable; (ii) every planar graph with neither 4-, 5-, 7-cycles nor triangles at distance less than two is DP-3-colorable. 
\end{abstract} 

\section{Introduction}
All graphs considered in this paper are finite, simple and undirected. A planar graph is a graph that can be embedded into the plane so that its edges meet only at their ends. A plane graph is a particular embedding of a planar graph into the plane. We set a plane graph $G = (V, E, F)$ where $V, E, F$ are the sets of vertices, edges, and faces of $G$, respectively. A vertex $v$ and a face $f$ are {\bf incident} if $v \in V(f)$. Two faces are {\bf adjacent} if they have at least one common edge. Call $v\in V(G)$ a $k$-vertex, or a $k^{+}$-vertex, or a $k^{-}$-vertex if its degree is equal to $k$, or at
least $k$, or at most $k$, respectively. The notions of a $k$-face, a $k^{+}$-face and a $k^{-}$-face are similarly defined.

A proper $k$-coloring of a graph $G$ is a mapping $f$: $V(G) \longrightarrow[k]$ such that $f(u)\neq f(v)$ whenever $uv\in E(G)$, where $[k]=\{1, 2, \dots, k\}$. The smallest integer $k$ such that $G$ has a proper $k$-coloring is called the {\bf chromatic number} of $G$, denoted by $\chi(G)$. Vizing \cite{MR0498216}, and independently Erd\H{o}s, Rubin, and Taylor \cite{MR593902} introduced list-coloring as a generalization of proper coloring. A {\bf list-assignment} $L$ gives each vertex $v$ a list of available colors $L(v)$. A graph $G$ is {\bf $L$-colorable} if there is a proper coloring $\phi$ of $G$ such that $\phi(v) \in L(v)$ for each $v\in V(G)$. A graph $G$ is $k$-choosable if $G$ is $L$-colorable for each $L$ with $\vert L(v)\vert \geq k$. The minimum integer $k$ such that $G$ is $k$-choosable is called the {\bf list-chromatic number} $\chi_{\ell}(G)$. 

For ordinary coloring, since every vertex has the same color set $[k]$, the operation of vertex identification is allowed. For list-coloring, the vertices may have different lists, so it is infeasible to identify vertices in general. To overcome this difficulty, Dvo\v{r}\'{a}k and Postle \cite{MR3758240} introduced DP-coloring under the name ``correspondence coloring'', showing that every planar graph without cycles of lengths 4 to 8 is 3-choosable. 

\begin{definition}\label{DEF1}
Let $G$ be a simple graph and $L$ be a list-assignment for $G$. For each vertex $v \in V(G)$, let $L_{v} = \{v\} \times L(v)$; for each edge $uv \in E(G)$, let $\mathscr{M}_{uv}$ be a matching between the sets $L_{u}$ and $L_{v}$, and let $\mathscr{M} = \bigcup_{uv \in E(G)}\mathscr{M}_{uv}$, called the {\bf matching assignment}. The matching assignment is called {\bf $k$-matching assignment} if $L(v) = [k]$ for each $v \in V(G)$. A {\bf cover} of $G$ is a graph $H_{L, \mathscr{M}}$ (simply write $H$) satisfying the following two conditions: 
\begin{enumerate}[label =(C\arabic*)]
\item the vertex set of $H$ is the disjoint union of $L_{v}$ for all $v \in V(G)$; 
\item the edge set of $H$ is the matching assignment $\mathscr{M}$.
\end{enumerate}
\end{definition}
Note that the matching $\mathscr{M}_{uv}$ is not required to be a perfect matching between the sets $L_{u}$ and $L_{v}$, and possibly it is empty. The induced subgraph $H[L_{v}]$ is an independent set for each vertex $v \in V(G)$. 

\begin{definition}
Let $G$ be a simple graph and $H$ be a cover of $G$. An {\bf $\mathscr{M}$-coloring} of $H$ is an independent set $\mathcal{I}$ in $H$ such that $\vert\mathcal{I} \cap L_{v}\vert = 1$ for each vertex $v \in V(G)$. The graph $G$ is {\bf DP-$k$-colorable} if for any list-assignment $\vert L(v)\vert \geq k$ and any matching assignment $\mathscr{M}$, it has an $\mathscr{M}$-coloring. The {\bf DP-chromatic number $\chi_{\mathrm{DP}}(G)$} of $G$ is the least integer $k$ such that $G$ is DP-$k$-colorable. 
\end{definition}

We mainly concentrate on DP-coloring of planar graphs in this paper. Dvo\v{r}\'{a}k and Postle \cite{MR3758240} noticed that $\chi_{\mathrm{DP}}(G)\leq5$ if $G$ is a planar graph, and $\chi_{\mathrm{DP}}(G)\leq3$ if $G$ is a planar graph with girth at least five. Several groups have given sufficient conditions for a planar graph to be DP-$3$-colorable, which extends the $3$-choosability of such graphs. 
\begin{theorem}[Liu \etal \cite{MR3886261}]
A planar graph is DP-$3$-colorable if it satisfies one of the following conditions:
\begin{enumerate}[label = (\arabic*)]
\item it contains no $3, 6, 7, 8$-cycles. 
\item it contains no $3, 5, 6$-cycles. 
\item it contains no $4, 5, 6, 9$-cycles. 
\item it contains no $4, 5, 7, 9$-cycles. 
\item the distance of triangles is at least two and it contains no $5, 6, 7$-cycles. 
\end{enumerate}
\end{theorem}

\begin{theorem}[Liu \etal \cite{MR3969021}]
If $a$ and $b$ are distinct values from $\{6, 7, 8\}$, then every planar graph without $4$-, $a$-, $b$-, $9$-cycles is DP-$3$-colorable. 
\end{theorem}

Zhang and Wu \cite{MR2159447} showed that every planar graph without 4-, 5-, 6- and 9-cycles is 3-choosable. Zhang \cite{MR3033651} generalized this result by showing that every planar graph with neither adjacent triangles nor 5-, 6- and 9-cycles is 3-choosable. Liu \etal \cite{MR3886261} showed that every planar graph without 4-, 5-, 6- and 9-cycles is DP-3-colorable. In this paper, we first extend these results by showing the following theorem. 

\begin{restatable}{theorem}{MRESULTa}\label{MRESULTa}
Every planar graph with neither adjacent triangles nor 5-, 6- and 9-cycles is DP-$3$-colorable. 
\end{restatable}

The {\bf distance of two triangles} $T$ and $T'$ is defined as the value $\min\{\dist(x, y) : x \in T \mbox{ and } y \in T'\}$, where $\dist(x, y)$ is the distance of the two vertices $x$ and $y$. In general, we use $\dist^{\triangledown}$ to denote the minimum distance of two triangles in a graph. Yin and Yu \cite{MR3954054} gave the following Bordeaux condition for planar graphs to be DP-3-colorable. 
\begin{theorem}[Yin and Yu \cite{MR3954054}]\label{YY}
A planar graph is DP-$3$-colorable if it satisfies one of the following two conditions:
\begin{enumerate}[label = (\arabic*)]
\item the distance of triangles is at least three and it contains no $4, 5$-cycles. 
\item the distance of triangles is at least two and it contains no $4, 5, 6$-cycles. 
\end{enumerate}
\end{theorem}
\autoref{YY} implies the following new results on 3-choosability. 
\begin{corollary}
A planar graph is 3-choosable if it satisfies one of the following conditions:
\begin{enumerate}[label = (\arabic*)]
\item the distance of triangles is at least three and it contains no $4, 5$-cycles. 
\item the distance of triangles is at least two and it contains no $4, 5, 6$-cycles. 
\end{enumerate}
\end{corollary}

\begin{table}
\centering
\caption{List-3-coloring and DP-3-coloring.}
\label{T1}
\begin{threeparttable}[b]
\begin{tabular}{|c|c|c|c|c|c|c|c|c|r|r|}
\hline
$\dist^{\triangledown}$&$3$&$4$&$5$&$6$&$7$&$8$&$9$&$10$&list-3-coloring&DP-3-colorable\\
\hline\rowcolor{LGrey}
        &\xcancel{3}&\xcancel{4}&  &  &  &  &  &  &Thomassen, 1995 \cite{MR1328294}&Dvo\v{r}\'{a}k, Postle, 2018 \cite{MR3758240}\\
        &\xcancel{3}&  &\xcancel{5}&\xcancel{6}&  &  &  &  &Lam, Shiu, Song, 2005 \cite{MR2137572}&Liu \etal 2019 \cite{MR3886261}\\\rowcolor{LGrey}
        &\xcancel{3}&  &  &\xcancel{6}&\xcancel{7}&  &  &  &Dvo\v{r}\'{a}k \etal 2010 \cite{MR2680225}&\faQuestion\\
        &\xcancel{3}&  &  &\xcancel{6}&\xcancel{7}&  &\xcancel{9}&  &Zhang, Xu, 2004 \cite{MR2038769}&\\
\rowcolor{LGrey}
        &\xcancel{3}&  &  &\xcancel{6}&\xcancel{7}&\xcancel{8}&  &  &Lidick\'y, 2009 \cite{MR2527000}&Liu \etal 2019 \cite{MR3886261}\\
        &\xcancel{3}&  &  &  &\xcancel{7}&\xcancel{8}&  &  &Dvo\v{r}\'{a}k \etal 2009 \cite{MR2552620}&\faQuestion\\\rowcolor{LGrey}
        &  &\xcancel{4}&\xcancel{5}&\xcancel{6} &\xcancel{7} &\xcancel{8}&&&Dvo\v{r}\'{a}k, Postle, 2018 \cite{MR3758240}&\faQuestion\\
        &  &\xcancel{4}&\xcancel{5}&\xcancel{6}&  &  &\xcancel{9}&    &Zhang, Wu, 2005 \cite{MR2159447}&Liu \etal 2019 \cite{MR3886261}\\\rowcolor{LGrey}
        &  &\xcancel{4}&\xcancel{5}&  &\xcancel{7}&  &\xcancel{9}&    &Zhang, Wu, 2004 \cite{MR2098488}&Liu \etal 2019 \cite{MR3886261}\\
        &  &\xcancel{4}&\xcancel{5}&  &\xcancel{7}&  &  &\xcancel{10}&Zhang, 2012 \cite{MR2976370}&\faQuestion\\
\rowcolor{LGrey}
        &  &\xcancel{4}&\xcancel{5}&  &  &\xcancel{8}&\xcancel{9}&    &Wang, Lu, Chen, 2010 \cite{MR2558977}&\faQuestion\\
        &  &\xcancel{4}&  &\xcancel{6}&\xcancel{7}&  &\xcancel{9}&    &Wang, Lu, Chen, 2008 \cite{MR2381380}&Liu \etal 2019 \cite{MR3969021}\\\rowcolor{LGrey}
        &  &\xcancel{4}&  &\xcancel{6}&  &\xcancel{8}&\xcancel{9}&    &Shen, Wang, 2007 \cite{MR2349483}&Liu \etal 2019 \cite{MR3969021}\\
        &  &\xcancel{4}&  &  &\xcancel{7}&\xcancel{8}&\xcancel{9}&    &Wang, Wu, Shen, 2011 \cite{MR2746963}&Liu \etal 2019 \cite{MR3969021}\\\rowcolor{LGrey}
        &  &\xcancel{4}&\xcancel{5}&  &  &\xcancel{8}&  &\xcancel{10}&Wang, Wu, 2011 \cite{MR2856858}&\faQuestion\\
$\geq 3$&  &\xcancel{4}&\xcancel{5}&  &  &  &  &    &derived from \cite{MR3954054}&Yin, Yu, 2019 \cite{MR3954054}\\
\rowcolor{LGrey}
$\geq 2$&  &\xcancel{4}&\xcancel{5}&\xcancel{6}&  &  &  &    &derived from \cite{MR3954054}&Yin, Yu, 2019 \cite{MR3954054}\\
$\geq 2$&  &  &\xcancel{5}&\xcancel{6}&\xcancel{7}&  &  &    &Li, Chen, Wang, 2016 \cite{MR3558005}&Liu \etal 2019 \cite{MR3886261}\\\rowcolor{LGrey}
$\geq 2$&  &  &\xcancel{5}&\xcancel{6}&  &\xcancel{8}&  &    &Zhang, Sun, 2008 \cite{MR2428692}&this paper\\
$\geq 3$&  &  &\xcancel{5}&\xcancel{6}&  &  &  &\xcancel{10}&Zhang, 2016 \cite{MR3492654}&\\
\rowcolor{LGrey}
$\geq 3$&  &\xcancel{4}&  &  &\xcancel{7}&  &\xcancel{9}&   &Li, Wang, 2016 \cite{Li2016}&\\
$\geq 2$&  &\xcancel{4}&\xcancel{5}&  &\xcancel{7}&  &  &   &Han, 2009 \cite{Han2009}&this paper\\
\rowcolor{LGrey}
\hline
\end{tabular}
\end{threeparttable}
\label{TABLE}
\end{table}


The following are two Bordeaux-type results on $3$-choosability. 
\begin{theorem}[Zhang and Sun \cite{MR2428692}]\label{ZHS}
Every planar graph with neither 5-, 6-, 8-cycles nor triangles at distance less than two is 3-choosable.
\end{theorem}

\begin{theorem}[Han \cite{Han2009}]\label{Han}
Every planar graph with neither 4-, 5-, 7-cycles nor triangles at distance less than two is 3-choosable. 
\end{theorem}

In the last section, we give two Bordeaux-type results on DP-3-coloring. The first one improves \autoref{ZHS} and the second one improves \autoref{Han}. 
\begin{restatable}{theorem}{MRESULTb}\label{MRESULTb}
Every planar graph with neither 5-, 6-, 8-cycles nor triangles at distance less than two is DP-3-colorable.
\end{restatable}

\begin{theorem}\label{NO457}
Every planar graph with neither 4-, 5-, 7-cycles nor triangles at distance less than two is DP-3-colorable. 
\end{theorem}

It is observed that every $k$-degenerate graph is DP-$(k + 1)$-colorable. \autoref{NO457} can be derived from the following \autoref{2D}. 

\begin{restatable}{theorem}{MRESULTc}\label{2D}
Every planar graph with neither 4-, 5-, 7-cycles nor triangles at distance less than two is 2-degenerate. 
\end{restatable}

For more results on DP-coloring of planar graphs, we refer the reader to \cite{MR3983123,Lu2019,MR3881665,MR3802151,Li2019a,MR3996735}. For convenience, we collect some results on list-$3$-coloring and DP-$3$-coloring in \autoref{TABLE}. If $uv$ is incident with a $7^{+}$-face and a $4^{-}$-face, then we say $uv$ {\bf controls} the $4^{-}$-face. Similarly, if $uv$ is on a $7^{+}$-cycle and a $4^{-}$-cycle, then we say $uv$ {\bf controls} the $4^{-}$-cycle. A vertex $v$ on a $7^{+}$-face $f$ is {\bf rich} to $f$ if none of the two incident edges on $f$ control a $4^{-}$-face, {\bf semi-rich} if exactly one of the two incident edges on $f$ controls a $4^{-}$-face, and {\bf poor} if they control two $4^{-}$-faces. A $3$-vertex $v$ is {\bf weak} if $v$ is incident with a 3-face, {\bf semi-weak} if $v$ is incident with a 4-face, and {\bf strong} if $v$ is incident with no $4^-$-face. 

For a face $f \in F$, if all the vertices on $f$ in a cyclic order are $v_{1}$, $v_{2}$, $\dots$, $v_{k}$, then we write $f=v_{1}v_{2}\dots v_{k}$, and call $f$ a $\big(d(v_{1}), d(v_{2}), \dots, d(v_{k})\big)$-face. A face is called a {\bf $k$-regular face} if every vertex incident with it is a $k$-vertex. A ($d_{1}$, $d_{2}$, $\dots$, $d_{t}$)-path $v_{1}v_{2}\dots v_{t}$ on a face $g$ is a set of consecutive vertices along the facial walk of $g$ such that $d(v_{i})=d_{i}$ and the vertices are different. The notions of $d^{+}$ (or $d^{-}$) are similarly for $d(v)\geq d$ (or $d(v)\leq d$). 
\section{Preliminary}
In this short section, some preliminary results are given, and these results can be used separately elsewhere. Liu \etal \cite{MR3886261} showed the ``nearly $(k - 1)$-degenerate'' subgraph is reducible for DP-$k$-coloring. 
\begin{lemma}[Liu \etal \cite{MR3886261}]\label{DP-GREEDY}
Let $k \geq 3$, $K$ be a subgraph of $G$ and $G' = G - V(K)$. If the vertices of $K$ can be ordered as $v_{1}, v_{2}, \dots, v_{t}$ such that the following hold: 
\begin{enumerate}[label = (\arabic*)]
\item $\vert V(G') \cap N_{G}(v_{1})\vert < \vert V(G') \cap N_{G}(v_{t})\vert$; 
\item $d_{G}(v_{t}) \leq k$ and $v_{1}v_{t} \in E(G)$; 
\item for each $2 \leq i \leq t - 1$, $v_{i}$ has at most $k - 1$ neighbors in $G - \{v_{i+1}, v_{i+2}, \dots, v_{t}\}$, 
\end{enumerate}
then any DP-$k$-coloring of $G'$ can be extended to a DP-$k$-coloring of $G$. 
\end{lemma}

A graph is {\bf minimal non-DP-$k$-colorable} if it is not DP-$k$-colorable but every subgraph with fewer vertices is DP-$k$-colorable. We give more specific reducible ``nearly $2$-degenerate'' configuration for DP-$3$-coloring. 
\begin{lemma}\label{Reducible}
Suppose that $G$ is a minimal non-DP-$3$-colorable graph and it has no adjacent $4^{-}$-cycles. Let $\mathcal{C}$ be an $m$-cycle $v_{1}v_{2}\dots v_{m}$, let $X = \{\,i : d(v_{i}) = 4, \quad 1 \leq i \leq m\,\}$ and $E^{+} = \{\,v_{i}v_{i+1} : i \in X\,\} \cup \{v_{m}v_{1}\}$. If $v_{m}$ is a $3$-vertex and $v_{m}v_{1}$ controls a $3$-cycle $v_{m}v_{1}u$ or a $4$-cycle $v_{m}v_{1}uw$, then $G$ contains no configuration satisfying all the following conditions:
\begin{enumerate}[label = (\roman*)]
\item every edge $e$ in $E^{+}$ controls a $4^{-}$-cycle $C_{e}$; 
\item all the vertices on $\mathcal{C}$ and the other vertices on cycles controlled by $E^{+}$ are distinct; 
\item every vertex on $\mathcal{C}$ is a $4^{-}$-vertex;
\item every vertex on cycles controlled by $E^{+}$ but not on $\mathcal{C}$ is a $3$-vertex; 
\item the vertex $u$ has a neighbor neither on $\mathcal{C}$ nor on the cycles controlled by $E^{+}$.
\end{enumerate}
\end{lemma}
\begin{proof}
Suppose to the contrary that there exists such a configuration. For the path $P = v_{1}v_{2}\dots v_{m}$, replace each edge $v_{i}v_{i+1}$ in $E(P) \cap E^{+}$ by the other part of the controlled cycle, and append $v_{m}u$ (when $v_{m}v_{1}u$ is the controlled $3$-cycle) or $v_{m}wu$ (when $v_{m}v_{1}uw$ is the controlled $4$-cycle) at the end. This yields a path starting at $v_{1}$ and ending at $u$. This path trivially corresponds to a sequence of vertices, and the sequence satisfies the condition of \autoref{DP-GREEDY} with $k = 3$, a contradiction. 
\end{proof}

By the definition of minimal non-DP-$k$-colorable, it is easy to obtain the following lemma. 
\begin{theorem}\label{delta}
If $G$ is a minimal non-DP-$k$-colorable graph, then $\delta(G) \geq k$. 
\end{theorem}

The following structural result for minimal non-DP-$k$-colorable graphs is a consequences of Theorem in \cite{Lu2019a}. 
\begin{theorem}\label{NONDP}
Let $G$ be a graph and $B$ be a $2$-connected induced subgraph of $G$ with $d_{G}(v) = k$ for all $v \in V(B)$. If $G$ is a minimal non-DP-$k$-colorable graph, then $B$ is a cycle or a complete graph.   
\end{theorem}

\section{Proof of \autoref{MRESULTa}}
Recall that our first main result is the following. 
\MRESULTa*
\begin{proof}
Let $G$ be a counterexample to the theorem with fewest number of vertices. We may assume that $G$ has been embedded in the plane. Thus, it is a minimal non-DP-$3$-colorable graph with $\delta(G) \geq 3$, and 
\begin{enumerate}[label = (\arabic*)]
\item $G$ is connected;
\item $G$ is a plane graph without adjacent triangles and 5-, 6-, 9-cycles; 
\item $G$ is not DP-3-colorable;
\item any subgraph with fewer vertices is DP-3-colorable. 
\end{enumerate}

A {\bf poor face} is a $10$-face incident with ten $3$-vertices, adjacent to one 4-face and four 3-faces. A {\bf bad face} is a $10$-face incident with ten $3$-vertices and adjacent to five 3-faces. A {\bf bad vertex} is a $3$-vertex on a bad face. A {\bf bad edge} is an edge on the boundary of a bad face. A {\bf special face} is a $(3, 3, 3, 3, 3, 4, 3, 3, 4, 3)$-face adjacent to six 3-faces. A {\bf semi-special face} is a $(3, 3, 3, 3, 3, 4, 3, 3, 4, 3)$-face adjacent to five 3-faces and one $4$-face as depicted in \autoref{fig:subfig:d-}. An illustration of these faces is in \autoref{TENFACE}. 
\begin{figure}%
\centering
\subcaptionbox{poor face \label{fig:subfig:a-}}[0.2\linewidth]
{
\begin{tikzpicture}
\coordinate (A) at (108:1);
\coordinate (B) at (144:1);
\coordinate (C) at (180:1);
\coordinate (D) at (216:1);
\coordinate (E) at (252:1);
\coordinate (F) at (288:1);
\coordinate (G) at (324:1);
\coordinate (H) at (0:1);
\coordinate (I) at (36:1);
\coordinate (J) at (72:1);

\filldraw[fill=gray] (A)--(B)--(C)--(D)--(E)--(F)--(G)--(H)--(I)--(J)--cycle;
\draw (B)--(162:1.2)--(C);
\draw (D)--(234:1.2)--(E);
\draw (F)--(306:1.2)--(G);
\draw (H)--(18:1.2)--(I);
\draw (J)--(80:1.3)--(100:1.3)--(A);
\fill
(A) circle (2pt)
(B) circle (2pt)
(C) circle (2pt)
(D) circle (2pt)
(E) circle (2pt)
(F) circle (2pt)
(G) circle (2pt)
(H) circle (2pt)
(I) circle (2pt)
(J) circle (2pt)
(162:1.2) circle (2pt)
(234:1.2) circle (2pt)
(306:1.2) circle (2pt)
(18:1.2) circle (2pt)
(80:1.3) circle (2pt)
(100:1.3) circle (2pt);
\end{tikzpicture}}
\subcaptionbox{bad face \label{fig:subfig:b-}}[0.2\linewidth]
{
\begin{tikzpicture}
\coordinate (A) at (108:1);
\coordinate (B) at (144:1);
\coordinate (C) at (180:1);
\coordinate (D) at (216:1);
\coordinate (E) at (252:1);
\coordinate (F) at (288:1);
\coordinate (G) at (324:1);
\coordinate (H) at (0:1);
\coordinate (I) at (36:1);
\coordinate (J) at (72:1);

\filldraw[fill=gray] (A)--(B)--(C)--(D)--(E)--(F)--(G)--(H)--(I)--(J)--cycle;
\draw (B)--(162:1.2)--(C);
\draw (D)--(234:1.2)--(E);
\draw (F)--(306:1.2)--(G);
\draw (H)--(18:1.2)--(I);
\draw (J)--(90:1.2)--(A);
\fill
(A) circle (2pt)
(B) circle (2pt)
(C) circle (2pt)
(D) circle (2pt)
(E) circle (2pt)
(F) circle (2pt)
(G) circle (2pt)
(H) circle (2pt)
(I) circle (2pt)
(J) circle (2pt)
(162:1.2) circle (2pt)
(234:1.2) circle (2pt)
(306:1.2) circle (2pt)
(18:1.2) circle (2pt)
(90:1.2) circle (2pt);
\end{tikzpicture}}
\subcaptionbox{special face \label{fig:subfig:c-}}[0.2\linewidth]
{
\begin{tikzpicture}
\coordinate (A) at (108:1);
\coordinate (B) at (144:1);
\coordinate (C) at (180:1);
\coordinate (D) at (216:1);
\coordinate (E) at (252:1);
\coordinate (F) at (288:1);
\coordinate (G) at (324:1);
\coordinate (H) at (0:1);
\coordinate (I) at (36:1);
\coordinate (J) at (72:1);

\filldraw[fill=gray] (A)--(B)--(C)--(D)--(E)--(F)--(G)--(H)--(I)--(J)--cycle;
\draw (A)--(126:1.2)--(B)--(162:1.2)--(C);
\draw (D)--(234:1.2)--(E);
\draw (F)--(306:1.2)--(G);
\draw (H)--(18:1.2)--(I)--(54:1.2)--(J);
\fill
(A) circle (2pt)
(B) circle (2pt)
(C) circle (2pt)
(D) circle (2pt)
(E) circle (2pt)
(F) circle (2pt)
(G) circle (2pt)
(H) circle (2pt)
(I) circle (2pt)
(J) circle (2pt)
(126:1.2) circle (2pt)
(162:1.2) circle (2pt)
(234:1.2) circle (2pt)
(306:1.2) circle (2pt)
(18:1.2) circle (2pt)
(54:1.2) circle (2pt)
;
\end{tikzpicture}}
\subcaptionbox{semi-special face \label{fig:subfig:d-}}[0.2\linewidth]
{
\begin{tikzpicture}
\coordinate (A) at (108:1);
\coordinate (B) at (144:1);
\coordinate (C) at (180:1);
\coordinate (D) at (216:1);
\coordinate (E) at (252:1);
\coordinate (F) at (288:1);
\coordinate (G) at (324:1);
\coordinate (H) at (0:1);
\coordinate (I) at (36:1);
\coordinate (J) at (72:1);

\filldraw[fill=gray] (A)--(B)--(C)--(D)--(E)--(F)--(G)--(H)--(I)--(J)--cycle;
\draw (A)--(126:1.2)--(B)--(152:1.3)--(172:1.3)--(C);
\draw (D)--(234:1.2)--(E);
\draw (F)--(306:1.2)--(G);
\draw (H)--(18:1.2)--(I)--(54:1.2)--(J);
\fill
(A) circle (2pt)
(B) circle (2pt)
(C) circle (2pt)
(D) circle (2pt)
(E) circle (2pt)
(F) circle (2pt)
(G) circle (2pt)
(H) circle (2pt)
(I) circle (2pt)
(J) circle (2pt)
(126:1.2) circle (2pt)
(152:1.3) circle (2pt)
(172:1.3) circle (2pt)
(234:1.2) circle (2pt)
(306:1.2) circle (2pt)
(18:1.2) circle (2pt)
(54:1.2) circle (2pt)
;
\end{tikzpicture}}
\caption{Some $10$-faces in \autoref{MRESULTa}.}
\label{TENFACE}
\end{figure}
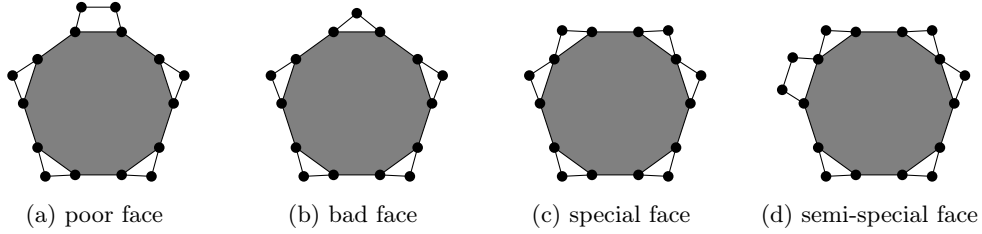

\begin{figure}%
\centering
\subcaptionbox{A $7$-face adjacent to a $3$-face \label{fig:subfig:71}}[0.4\linewidth]
{
\begin{tikzpicture}
\coordinate (A) at (30:1.1);
\coordinate (B) at (80:1.1);
\coordinate (C) at (130:1.1);
\coordinate (D) at (180:1.1);
\coordinate (E) at (230:1.1);
\coordinate (F) at (280:1.1);
\coordinate (G) at (330:1.1);
\coordinate (H) at (0:1.6);
\filldraw[fill=gray] (A)--(B)--(C)--(D)--(E)--(F)--(G)--cycle;
\filldraw[fill=green] (A)--(H)--(G)--cycle;
\fill
(A) circle (2pt)
(B) circle (2pt)
(C) circle (2pt)
(D) circle (2pt)
(E) circle (2pt)
(F) circle (2pt)
(G) circle (2pt)
(H) circle (2pt);
\end{tikzpicture}}
\subcaptionbox{A $7$-face adjacent to a $4$-face \label{fig:subfig:72}}[0.4\linewidth]
{
\begin{tikzpicture}
\coordinate (O) at (0, 0);
\coordinate (A) at (45:3);
\coordinate (B) at (135:3);
\coordinate (A') at (60:1);
\coordinate (B') at (120:1);
\coordinate (C) at (90:1.5);
\coordinate (C') at (90:1);
\filldraw[even odd rule, fill=gray] (O)--(A)--(B)--(O)--(A')--(C)--(B')--cycle;
\filldraw[fill=green] (O)--(A')--(C')--(B')--cycle;
\fill
(O) circle (2pt)
(A) circle (2pt)
(B) circle (2pt)
(A') circle (2pt)
(B') circle (2pt)
(C) circle (2pt)
(C') circle (2pt);
\draw (C')--(90:1.2);
\end{tikzpicture}}
\caption{A $7$-face adjacent to a $4^{-}$-face.}
\label{7-4}
\end{figure}
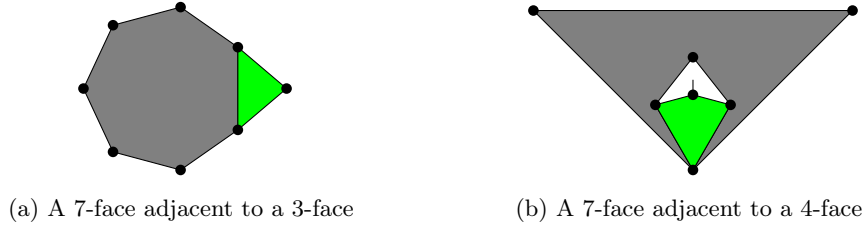

By \autoref{NONDP}, we can easy obtain the following structural result. 
\begin{lemma}\label{B}
Let $f$ be a 10-face bounded by a cycle in $G$. If $f$ is incident with ten $3$-vertices and it controls a $4^{-}$-face, then the controlled $4^{-}$-face is incident with at least one $4^{+}$-vertex. 
\end{lemma}

By \autoref{B} and the definitions of poor faces and bad faces, we have the following consequences. 
\begin{lemma}\label{FORBIDDEN}
\mbox{}
\begin{enumerate}[label = (\roman*)]
\item There are no adjacent poor faces. 
\item There are no adjacent bad faces. 
\item There are no poor faces adjacent to bad faces. 
\end{enumerate}
\end{lemma}

The following structural results will be frequently used. 
\begin{lemma}\label{LS}\mbox{}
\begin{enumerate}[label=(\alph*)]
\item\label{a} Every $7^{-}$-cycle is chordless. 
\item\label{b} Every $3$-cycle is not adjacent to $6^{-}$-cycle. 
\item\label{c} Every $7$-face is adjacent to at most one $4^{-}$-face; the possible situations see \autoref{7-4}. Consequently, there are no bad faces adjacent to $7$-faces. 
\item\label{d} No $8$-face is adjacent to a $3$-face; no $9$-face is adjacent to a $3$-face. 
\item\label{e} There are no adjacent $6^{-}$-faces; thus every $3$-vertex is adjacent to at most one $4^{-}$-face. 
\end{enumerate}
\end{lemma}
\begin{proof}
\ref{a} If a $4$-cycle has a chord, then there are two adjacent triangles. Note that $5$-cycles and $6$-cycles are excluded in $G$. If a $7$-cycle has a chord, then there is a $5$- or $6$-cycle, a contradiction. 

\ref{b} Note that $5$-cycles and $6$-cycles are excluded in $G$. If a $3$-cycle is adjacent to a $3$- or $4$-cycle, then it contradicts \autoref{LS}\ref{a}. 

\ref{c} Let $f$ be a 7-face and $C$ be its boundary. (i) Suppose that $C$ is a cycle. If $w_{1}w_{2}w_{3}w_{4}$ is on the boundary and $w_{2}w_{3}$ is incident with a $4$-face $u_{1}w_{2}w_{3}u_{4}$, then none of $u_{1}$ and $u_{4}$ is on $C$ because $C$ is chordless and $\delta(G) \geq 3$, but $C$ and $u_{1}w_{2}w_{3}u_{4}$ form a $9$-cycle, a contradiction. Suppose that $f$ is adjacent to two $3$-faces $uvw$ and $u'v'w'$ with $uv, u'v'$ on $C$. If $w = w'$, then there are two adjacent triangles or a $5$-cycle, a contradiction; and if $w \neq w'$, then there is a $9$-cycle, a contradiction. (ii) Suppose that $C$ is not a cycle, and thus it consists of a $3$-cycle and a $4$-cycle. Hence, $f$ cannot be adjacent to any $3$-face by \autoref{LS}\ref{b}. If $f$ is adjacent to a $4$-face, then it can only be shown in \autoref{fig:subfig:72}. Therefore, $f$ is adjacent to at most one $4^{-}$-face. 

\ref{d} If an $8$-face is bounded by a cycle, then it cannot be adjacent to a $3$-face, otherwise they form a $9$-cycle or a $8$-cycle with two chords, a contradiction. Suppose that the boundary of an $8$-face is not a cycle but it is adjacent to a $3$-face. By \autoref{LS}\ref{b}, the boundary of the $8$-face must contain a $7^{+}$-cycle, but this is impossible. 

Since there is no $9$-cycle, the boundary of a $9$-face is not a cycle. Suppose the boundary of a $9$-face is adjacent to a $3$-face. By \autoref{LS}\ref{b}, the boundary of the $9$-face must contain a $7^{+}$-cycle, but this is impossible. 

\ref{e} Since there is no $6$-cycle, the boundary of a $6$-face consists of two triangles. It is easy to check that there are no adjacent $6^{-}$-faces. 
\end{proof}

\begin{lemma}\label{4-poor}
Each $(3, 3, 3^{+}, 4^{+})$-face $f$ is adjacent to at most one poor face.
\end{lemma}
\begin{proof}
Since every poor face is incident with ten $3$-vertices, $f$ can only be adjacent to poor faces via $(3, 3)$-edges. Let $f=v_{1}v_{2}v_{3}v_{4}$ with $d(v_{1})=d(v_{2})=3$, $d(v_{3})\geq3$ and $d(v_{4})\geq4$. If $d(v_{3})\geq4$, then $f$ is incident with exactly one $(3, 3)$-edge, and then it is adjacent to at most one poor face. Suppose that $d(v_{3})=3$ and $f$ is adjacent to two poor faces $f_{1}$ and $f_{2}$ via $v_{1}v_{2}$ and $v_{2}v_{3}$. Since $v_{2}$ is a $3$-vertex, the poor face $f_{1}$ is adjacent to the poor face $f_{2}$, but this contradicts \autoref{FORBIDDEN}. 
\end{proof}

\begin{lemma}\label{BAD-SPECIAL}
Each bad face is adjacent to at most two special faces.
\end{lemma}
\begin{proof}
Let $f=v_{1}v_{2}\dots v_{10}$ be a bad face and incident with five $3$-faces $v_{1}v_{2}u_{1}, v_{3}v_{4}u_{3}, v_{5}v_{6}u_{5}, v_{7}v_{8}u_{7}, v_{9}v_{10}u_{9}$. Suppose that $f$ is adjacent to $f_{i}$ via edge $v_{i}v_{i+1}$ for $1 \leq i \leq 10$, where the subscripts are taken modulo $10$. Suppose to the contrary that $f$ is adjacent to at least three special faces. Then there exist two special faces $f_{m}$ and $f_{n}$ such that $\vert m - n\vert = 2$ or 8, where $\{m, n\} \subset \{2, 4, 6, 8, 10\}$. Without loss of generality, assume that $f_{2}$ and $f_{4}$ are the two special faces. By \autoref{Reducible} and the definition of special face, $d(u_{1}) = d(u_{3}) = 4$. Let $x_{3}$ and $x_{4}$ be the neighbors of $u_{3}$ other than $v_{3}$ and $v_{4}$. Since $f_{2}$ and $f_{4}$ are special faces, we have that $x_{3}x_{4} \in E(G)$ and $d(x_{3}) = d(x_{4}) = 3$, but this contradicts \autoref{Reducible}. 
\end{proof}

\begin{lemma}\label{BADFACE}
Suppose that $f$ is a $10^{+}$-face and it is not a bad face. Let $t$ be the number of incident bad edges, and $t \geq 1$. Then $3t \leq d(f)$. Moreover, if $d(f) > 3t$, then $f$ is incident with at least $(t + 1)$ $4^{+}$-vertices (repeated vertices are counted as the number of appearance on the boundary). 
\end{lemma}
\begin{proof}
Suppose that $f$ is adjacent to a bad face through $uv$. Let $x$ be the neighbor of $u$ on $f$ and $y$ be the neighbor of $v$ on $f$. Then $u$ and $v$ are bad vertices and the faces controlled by $f$ through $xu$ and $vy$ are all 3-faces. By \autoref{B} and the definition of bad face, $d(x) \geq 4$ and $d(y) \geq4$. It is observed that two bad edges are separated by at least two other edges along the boundary of $f$, this implies that $3t \leq d(f)$. 

By the above discussion, every bad vertex has a $4^{+}$-neighbor along the boundary. Since $3t < d(f)$, there are two bad edges separated by at least two $4^{+}$-vertices, thus $f$ is incident with at least $(t + 1)$ $4^{+}$-vertices. 
\end{proof}

To prove the theorem, we are going to use discharging method. Define the initial charge function $\mu(x)$ on $V \cup F$ to be $\mu(v)=d(v)-6$ for $v\in V$ and $\mu(f)=2d(f)-6$ for $f\in F$. By Euler's formula, we have the following equality, 
\[
\sum_{v\in V(G)}(d(v)-6) + \sum_{f\in F(G)}(2d(f)-6) = -12.
\]

We design suitable discharging rules to change the initial charge function $\mu(x)$ to the final charge function $\mu'(x)$ on $V\cup F$ such that $\mu'(x)\geq0$ for all $x\in V\cup F$, this leads to a contradiction and completes the proof. 

The following are the needed discharging rules.

\begin{enumerate}[label = \bf R\arabic*]
\item Each $4$-face sends $\frac{1}{2}$ to each incident $3$-vertex.
\item Each $6$-face sends $1$ to each incident vertex. 
\item Each $7$-face sends $\frac{3}{2}$ to each incident semi-rich $3$-vertex, $1$ to each other incident vertex.  
\item Each $8$-face sends $\frac{5}{4}$ to each incident vertex. 
\item Each $9$-face sends $\frac{4}{3}$ to each incident vertex. 
\item Suppose that $v$ is a $3$-vertex incident with a $10^{+}$-face $f$ and two other faces $g$ and $h$. 
\begin{enumerate}[label = \bf (\alph*)]
\item If $v$ is incident with three $5^{+}$-faces, then $f$ sends $1$ to $v$. 
\item If $v$ is incident with a $4$-face, then $f$ sends $\frac{5}{4}$ to $v$; 
\item If $f$ is a bad face, $g$ is a $3$-face and $h$ is not a special face, then $f$ sends $\frac{4}{3}$ to $v$ and $h$ sends $\frac{5}{3}$ to $v$. 
\item Otherwise, $f$ sends $\frac{3}{2}$ to $v$. 
\end{enumerate}
\item Let $v$ be a $4$-vertex on a $10^{+}$-face $f$. 
\begin{enumerate}[label = \bf (\alph*)]
\item If $v$ is a rich vertex or a poor vertex of $f$, then $f$ sends $1$ to $v$. 
\item Otherwise, $f$ sends $\frac{1}{2}$ to $v$. 
\end{enumerate}
\item Let $v$ be a 5-vertex on a $10^{+}$-face $f$. 
\begin{enumerate}[label = \bf (\alph*)]
\item If $v$ is incident with two $4^{-}$-face, then $f$ sends $\frac{1}{3}$ to $v$. 
\item If $v$ is incident with exactly one $4^{-}$-face, then $f$ sends $\frac{1}{4}$ to $v$. 
\item Otherwise, $f$ sends $\frac{1}{5}$ to $v$. 
\end{enumerate}
\item Each $(3, 3, 3^{+}, 4^{+})$-face sends $\frac{1}{2}$ to adjacent poor face.
\item Each $(3, 4, 3^{+}, 4^{+})$-face and $(3, 4, 4^{+}, 3^{+})$-face send $\frac{1}{4}$ to each adjacent semi-special face.
\end{enumerate}

It remains to check that the final charge of every element in $V \cup F$ is nonnegative.

(1) Let $v$ be an arbitrary vertex of $G$. 

By \autoref{delta}, $G$ has no $2^{-}$-vertices. If $v$ is a $6^{+}$-vertex, then $\mu'(v) \geq \mu(v) = d(v)-6\geq0$. We may assume that $3 \leq d(v)\leq 5$.

Suppose that $v$ is a $3$-vertex. By \autoref{LS}\ref{e}, $v$ is incident with at most one $4^{-}$-face. If $v$ is incident with no $4^{-}$-face, then it receives at least $1$ from each incident face, and then $\mu'(v) \geq 3 - 6 + 3 \times 1 = 0$. If $v$ is incident with a $4$-face, then it receives at least $\frac{5}{4}$ from each incident $7^{+}$-face, and then $\mu'(v) \geq 3 - 6 + 2 \times \frac{5}{4} + \frac{1}{2} = 0$. If $v$ is incident with a $3$-face and a $7$-face, then the other incident face is not a bad face by \autoref{LS}\ref{c}, and then $\mu'(v) = 3 - 6 + 2 \times \frac{3}{2} = 0$. By \autoref{LS}\ref{d}, if $v$ is incident with a $3$-face, then it is not incident with any $8$- or $9$-face. If $v$ is incident with a $3$-face and two $10^{+}$-faces, then $\mu'(v) \geq 3 - 6 + \min\big\{\frac{4}{3} + \frac{5}{3}, 2 \times \frac{3}{2}\big\} = 0$. 

Suppose that $v$ is a $4$-vertex. By \autoref{LS}\ref{e}, $v$ is incident with at most two $4^{-}$-faces. If $v$ is incident with exactly one $4^{-}$-face, then $\mu'(v)\geq 4 - 6 + 2 \times \frac{1}{2 } + 1 = 0$. If $v$ is incident with two $4^{-}$-faces, then $\mu'(v) \geq 4 - 6 + 2 \times 1 = 0$. If $v$ is incident with no $4^{-}$-face, then $\mu'(v) \geq 4 - 6 + 4 \times 1 > 0$.

Suppose that $v$ is a 5-vertex. By \autoref{LS}\ref{e}, $v$ is incident with at most two $4^{-}$-faces. If $v$ is incident with no $4^{-}$-face, then it receives at least $\frac{1}{5}$ from each incident $5^{+}$-face, and $\mu'(v)\geq 5 - 6 + 5 \times \frac{1}{5} = 0$. If $v$ is incident with exactly one $4^{-}$-face, then it receives at least $\frac{1}{4}$ from each incident $5^{+}$-face, and $\mu'(v) \geq 5 - 6 + 4 \times \frac{1}{4} = 0$. If $v$ is incident with two $4^{-}$-faces, then it receives at least $\frac{1}{3}$ from each incident $5^{+}$-face, and $\mu'(v) \geq 5 - 6 + 3 \times \frac{1}{3} = 0$.

(2) Let $f$ be an arbitrary face in $F(G)$. 

If $f$ is a 3-face, then $\mu'(f)=\mu(f) = 0$. Suppose that $f$ is a 4-face. If $f$ is incident with four $3$-vertices, then $\mu'(f) = 2 - 4 \times \frac{1}{2} = 0$. If $f$ is incident with exactly one $4^{+}$-vertex, then it is adjacent to at most one poor face by \autoref{4-poor}, and then $\mu'(f) \geq 2 - 3 \times \frac{1}{2} - \frac{1}{2} = 0$. If $f$ is a $(3, 3, 4^{+}, 4^{+})$-face, then it is adjacent to at most one poor face and at most two semi-special faces, and then $\mu'(f) \geq 2 - 2 \times \frac{1}{2} - \frac{1}{2} - 2 \times \frac{1}{4} = 0$. If $f$ is a $(3, 4^{+}, 3, 4^{+})$-face, then it sends at most $\frac{1}{4}$ to each adjacent face, and $\mu'(f) \geq 2 - 2 \times \frac{1}{2} - 4 \times \frac{1}{4} = 0$. If $f$ is incident with exactly three $4^{+}$-vertices, then it is adjacent to at most two semi-special faces, and $\mu'(f) \geq 2 - \frac{1}{2} - 2 \times \frac{1}{4} > 0$. If $f$ is incident with four $4^{+}$-vertices, then $\mu'(f) = \mu(f) = 2$.  

If $f$ is a $6$-face, then $\mu'(f) = 6 - 6 \times 1 = 0$. Suppose that $f$ is a 7-face. By \autoref{LS}\ref{c}, $f$ is adjacent to at most one $4^{-}$-face. If $f$ is adjacent to a $4^{-}$-face (see \autoref{7-4}), then $f$ is incident with at most two semi-rich $3$-vertices, which implies that $\mu'(f) \geq 8 - 2 \times \frac{3}{2} - 5 \times 1 = 0$. If $f$ is not adjacent to any $4^{-}$-face, then $f$ sends $1$ to each incident vertex, and $\mu'(f) = 8 - 7 \times 1 > 0$. If $f$ is an $8$-face, then $\mu'(f) = 10 - 8 \times \frac{5}{4} = 0$. If $f$ is a $9$-face, then $\mu'(f) = 12 - 9 \times \frac{4}{3} = 0$.

Suppose that $f$ is a $10^{+}$-face. Let $t$ be the number of incident bad edges. Hence, $f$ is incident with exactly $2t$ bad vertices. By \autoref{BADFACE}, $f$  is incident with at least $t$ $4^{+}$-vertices. Thus, $\mu'(f) \geq 2d(f) - 6 - 2t \times \frac{5}{3} - t \times 1-(d(f) - 3t) \times \frac{3}{2} = \frac{1}{2} d(f) - 6 + \frac{t}{6}$. If $d(f) \geq 12$, then $\mu'(f) \geq 12 \times \frac{1}{2} - 6 + \frac{t}{6}\geq0$. So it suffices to consider $10$-faces and $11$-faces. 

{\bf Suppose that $f$ is an $11$-face}. (i) $t=0$. It follows that $f$ is not incident with any bad vertex, and it sends at most $\frac{3}{2}$ to each incident vertex. If $f$ is incident with a $4^{+}$-vertex, then $\mu'(f) \geq 16 - 10 \times \frac{3}{2} - 1 = 0$. Suppose that $f$ is a $3$-regular face. By \autoref{LS}\ref{e}, every vertex on $f$ is incident with at most one $4^{-}$-face. Since $d(f)$ is odd, $f$ must be incident with a rich $3$-vertex. This implies that $\mu'(f) \geq 16 - 10 \times \frac{3}{2} - 1 = 0$. (ii) $t \geq 1$. It follows that $f$ is incident with exactly $2t$ bad vertices and at least $(t+1)$ $4^{+}$-vertices, and then $\mu'(f) \geq 16 - 2t \times \frac{5}{3} - (t + 1) \times 1 - (11 - (3t + 1)) \times\frac{3}{2} = \frac{t}{6} > 0$.

{\bf Finally we may assume that $f$ is a 10-face}. If $f$ is a special face, then $\mu'(f) = 14 - 8 \times \frac{3}{2} - 2 \times 1 = 0$. If $f$ is a bad face, then it is adjacent to at most two special faces by \autoref{BAD-SPECIAL}, which implies that $\mu'(f) \geq 14 - 4 \times \frac{3}{2} - 6 \times \frac{4}{3} = 0$. 

So we may assume that $f$ is neither a bad face nor a special face. By \autoref{BADFACE}, $t \leq \lfloor \frac{d(f)}{3} \rfloor = 3$.

$\bullet\,\, \bm{t = 0}$. It follows that $f$ is not incident with any bad vertex. Hence, $f$ sends at most $\frac{3}{2}$ to each incident $3$-vertex, at most $1$ to each incident $4$-vertex, and at most $\frac{1}{3}$ to each incident $5$-vertex. If $f$ is incident with a $5^{+}$-vertex, then $\mu'(f) \geq 14 - 9 \times \frac{3}{2} - \frac{1}{3} > 0$. If $f$ is incident with at least two $4$-vertices, then $\mu'(f) \geq 14 - 8 \times \frac{3}{2} - 2 \times 1 = 0$. 

So we may assume that $f$ is incident with at most one $4$-vertex and no $5^{+}$-vertices. If $f$ is incident with a semi-rich $4$-vertex, then $\mu'(f) \geq 14 - 9 \times \frac{3}{2} - \frac{1}{2} = 0$. If $f$ is incident with a rich $4$-vertex and nine $3$-vertices, then at least one of the incident $3$-vertices is rich, and then $\mu'(f) \geq 14 - 1 - 8 \times \frac{3}{2} - 1 = 0$. If $f$ is incident with a poor $4$-vertex, then there exists a rich $3$-vertex incident with $f$, and $\mu'(f) \geq 14 - 1 - 8 \times \frac{3}{2} - 1 = 0$. 

Suppose that $f$ is incident with ten $3$-vertices. If $f$ is adjacent to at most four $4^{-}$-faces, then $\mu'(f) \geq 14 - 8 \times \frac{3}{2} - 2 \times 1 = 0$. If $f$ is adjacent to at least two $4$-faces, then $\mu'(f) \geq 14 - 6 \times \frac{3}{2} - 4 \times \frac{5}{4} = 0$. If $f$ is adjacent to four $3$-faces and one $4$-face, then $f$ must be a poor face and the $4$-face must be $(3, 3, 3^{+}, 4^{+})$-face, and then $\mu'(f) = 14 - 8 \times \frac{3}{2} - 2 \times \frac{5}{4} + \frac{1}{2} = 0$. If $f$ is adjacent to five $3$-faces, then it is a bad face, so we are done. 

$\bullet\,\, \bm{t = 1}$. It follows that $f$ is incident with exactly two bad vertices and at least two $4^{+}$-vertices. If $f$ is incident with a rich $3$-vertex or at least three $4$-vertices, then $\mu'(f) \geq 14 - 2 \times \frac{5}{3} - 3 \times 1 - 5 \times \frac{3}{2} > 0$. If $f$ is incident with a $5^{+}$-vertex, then $\mu'(f) \geq 14 - 2 \times \frac{5}{3} - 1 - \frac{1}{3} - 6 \times \frac{3}{2} > 0$. If $f$ is incident with a semi-rich $4$-vertex, then $\mu'(f) \geq 14 - 2 \times \frac{5}{3} - 1 - \frac{1}{2} - 6 \times \frac{3}{2} > 0$. So we may assume that $f$ is incident with two poor $4$-vertices and eight semi-rich $3$-vertices. Thus $f$ is a $(3, 4, 3, 3, 4, 3, 3, 3, 3, 3)$-face $w_{1}w_{2} \dots w_{10}$, where $w_{3}w_{4}$ is a bad edge, each of $w_{2}w_{3}$ and $w_{4}w_{5}$ controls a $3$-face, each of $w_{1}w_{2}$ and $w_{5}w_{6}$ controls a $4^{-}$-face. If $f$ controls at least one $4$-face by a $(3, 3)$-edge, then $\mu'(f) \geq 14 - 2 \times \frac{5}{3} - 2 \times 1 - 2 \times \frac{5}{4} - 4 \times \frac{3}{2} > 0$. So we may further assume that each of $w_{7}w_{8}$ and $w_{9}w_{10}$ controls a $3$-face. If each of $w_{1}w_{2}$ and $w_{5}w_{6}$ controls a $4$-face, then $\mu'(f) \geq 14 - 2 \times \frac{5}{3} - 2 \times 1 - 2 \times \frac{5}{4} - 4 \times \frac{3}{2} > 0$. If each of $w_{1}w_{2}$ and $w_{5}w_{6}$ controls a $3$-face, then $f$ must be a special face, a contradiction. If exactly one of $w_{1}w_{2}$ and $w_{5}w_{6}$ controls a $4$-face, then $f$ must be a semi-special face. In this case the controlled $4$-face must be a $(3, 4, 3^{+}, 4^{+})$-face or $(3, 4, 4^{+}, 3^{+})$-face due to \autoref{Reducible}. Thus $\mu'(f) \geq 14 - 2 \times \frac{5}{3} - 2 \times 1 - \frac{5}{4} - 5\times\frac{3}{2} + \frac{1}{4} \geq 0$. 

$\bullet\,\, \bm{t = 2}$. It follows that $f$ is incident with exactly four bad vertices and at least three $4^{+}$-vertices. If $f$ is incident with at least four $4^{+}$-vertices, then $\mu'(f) \geq 14 - 4 \times \frac{5}{3} - 4 \times 1 - (10 - 4 - 4) \times \frac{3}{2} > 0$. Thus, $f$ is incident with exactly four bad vertices and exactly three $4^{+}$-vertices. If there is a semi-rich $4^{+}$-vertex, then $\mu'(f) \geq 14 - 4 \times \frac{5}{3} - 2 \times 1 - \frac{1}{2} - 3 \times \frac{3}{2} > 0$. Therefore, the three $4^{+}$-vertices are all poor, so there must be a rich $3$-vertex. This implies that $\mu'(f) \geq 14 - 4 \times \frac{5}{3} - 3 \times 1 - 2 \times \frac{3}{2} - 1 > 0$. 

$\bullet\,\, \bm{t = 3}$. It follows that $f$ is incident with six bad vertices and four $4^{+}$-vertices. Thus, $\mu'(f) \geq 14 - 6 \times \frac{5}{3} - 4 \times 1 =  0$. 
\end{proof}

\section{Distance of triangles at least two}
In this section, we give two Bordeaux type results on planar graphs with distance of triangles at least two. 
\subsection{Planar graphs without 5-, 6- and 8-cycles}
\begin{figure}%
\centering
\subcaptionbox{special face \label{fig:subfig:-a}}[0.4\linewidth]
{
\begin{tikzpicture}
\def\n{7}
\pgfmathsetmacro \i{\n-1}
\foreach \x in {0,...,\i}
{
\def\pointnameA{A\x}
\coordinate (\pointnameA) at ($(\x*360/\n + 90:1)$);
}

\filldraw[fill=gray] (A0)
\foreach \x in {0,...,\i}
 {-- (A\x)}  -- cycle;
 
\foreach \x in {0,...,\i}
{
\node[circle, inner sep =1, fill, draw] () at (A\x) {};
}

\foreach \x in {0,...,27}
{
\def\pointnameB{B\x}
\coordinate (\pointnameB) at ($(\x*90/\n + 90:1.4)$);
}

\draw (A0)--(B1)--(B3)--(A1);
\draw (A2)--(B9)--(B11)--(A3);
\draw (A4)--(B17)--(B19)--(A5);
\draw (A6)--(B25)--(B27)--(A0);
\node[circle, inner sep =1, fill, draw] () at (B1) {};
\node[circle, inner sep =1, fill, draw] () at (B3) {};
\node[circle, inner sep =1, fill, draw] () at (B9) {};
\node[circle, inner sep =1, fill, draw] () at (B11) {};
\node[circle, inner sep =1, fill, draw] () at (B17) {};
\node[circle, inner sep =1, fill, draw] () at (B19) {};
\node[circle, inner sep =1, fill, draw] () at (B25) {};
\node[circle, inner sep =1, fill, draw] () at (B27) {};
\end{tikzpicture}}
\subcaptionbox{poor face \label{fig:subfig:-b}}[0.4\linewidth]
{
\begin{tikzpicture}
\def\n{7}
\pgfmathsetmacro\i{\n-1}
\foreach \x in {0,...,\i}
{
\def\pointname{A\x}
\coordinate (\pointname) at ($(\x*360/\n + 90:1)$);
}

\filldraw[fill=gray] (A0)
\foreach \x in {0,...,\i}
 {-- (A\x)}  -- cycle;
 
\foreach \x in {0,...,\i}
{
\node[circle, inner sep =1, fill, draw] () at (A\x) {};
}

\foreach \y in {0,...,27}
{
\def\pointname{B\y}
\coordinate (\pointname) at ($(\y*90/\n + 90:1.4)$);
}
\draw (A0)--(B1)--(B3)--(A1);
\draw (A2)--(B9)--(B11)--(A3);
\draw (A4)--(B17)--(B19)--(A5);
\draw (A6)--(B24);
\node[circle, inner sep =1, fill, draw] () at (B1) {};
\node[circle, inner sep =1, fill, draw] () at (B3) {};
\node[circle, inner sep =1, fill, draw] () at (B9) {};
\node[circle, inner sep =1, fill, draw] () at (B11) {};
\node[circle, inner sep =1, fill, draw] () at (B17) {};
\node[circle, inner sep =1, fill, draw] () at (B19) {};
\node[circle, inner sep =1, fill, draw] () at (B24) {};
\end{tikzpicture}}
\caption{Certain $7$-faces in \autoref{MRESULTb}.}
\end{figure}
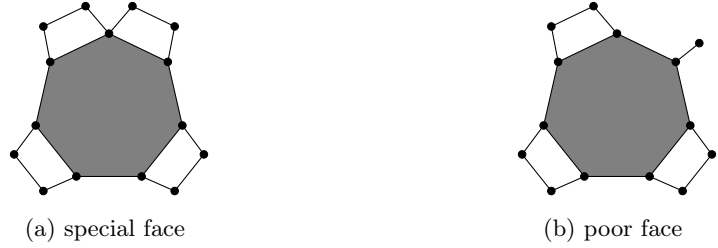

Recall that our second main result is the following. 
\MRESULTb*
\begin{proof}
Suppose to the contrary that $G$ is a counterexample with the number of vertices as small as possible. We may assume that $G$ has been embedded in the plane. Thus, $G$ is a minimal non-DP-$3$-colorable graph. 

\begin{lemma}\label{NO-ADJ}
\mbox{}
\begin{enumerate}[label = \bf(\alph*)]
\item\label{1} There are no $5$-faces and no $6$-faces. 
\item\label{2} A 3-face cannot be adjacent to an $8^{-}$-face.
\item\label{3} There are no adjacent $6^{-}$-faces. 
\end{enumerate}
\end{lemma}
\begin{proof}
Since every $5$-face is bounded by a $5$-cycle, but there is no $5$-cycles in $G$, this implies that there is no $5$-faces in $G$. Since there is no $6$-cycles in $G$, the boundary of every $6$-face consists of two triangles, thus the distance of these triangles is zero, a contradiction. Therefore, there is no $5$-faces and no $6$-faces in $G$. 

It is easy to check that every $7^{-}$-cycle is chordless. Let $f$ be a $3$-face. If $f$ is adjacent to a $4$-face $g$, then they form a $5$-cycle with a chord, a contradiction. Suppose that $g$ is a $7$-face. Then $g$ may be bounded by a cycle or a closed walk with a cut-vertex. If $g$ is bounded by a cycle and it is adjacent to $f$, then these two cycles form an $8$-cycle with a chord, a contradiction. If the boundary of $g$ contains a cut-vertex, then the boundary consists of a $3$-cycle and a $4$-cycle, and neither the $3$-cycle nor the $4$-cycle can be adjacent to the $3$-face $f$. If $g$ is an $8$-face, then the boundary of $g$ consists of two $4$-cycles, or two triangles and a cut-edge, but no edge on such boundary can be adjacent to the $3$-face $f$. 

By the hypothesis and fact that a $3$-face cannot be adjacent to an $8^{-}$-face, it suffices to prove that there is no adjacent $4$-faces. Since every $4$-cycle has no chords, two adjacent $4$-faces must form a $6$-cycle with a chord, a contradiction. 
\end{proof}

A 7-face $f$ is {\bf special} if $f$ is incident with six semi-weak $3$-vertices and a poor $4$-vertex, see \autoref{fig:subfig:-a}. A 7-face $f$ is {\bf poor} if $f$ is incident with six semi-weak $3$-vertices and a strong $3$-vertex, see \autoref{fig:subfig:-b}. Note that a $7$-face is not adjacent to any $3$-face by \autoref{NO-ADJ}\ref{2}. 

\begin{lemma}
Each poor 7-face is adjacent to three $(3, 3, 3^{+}, 4^{+})$-faces.
\end{lemma}
\begin{proof}
Suppose that $f = w_{1}w_{2}\dots w_{7}$ is a poor $7$-face and it is adjacent to a $4$-face $g = u_{1}w_{2}w_{3}u_{4}$. Since $f$ is incident with seven $3$-vertices, it must be bounded by a $7$-cycle. Note that every $7$-cycle has no chords, we have that $\{u_{1}, u_{2}\} \cap \{w_{1}, w_{2}, \dots, w_{7}\} = \emptyset$. The subgraph induced by $\{w_{1}, w_{2}, \dots, w_{7}\} \cup \{u_{1}, u_{2}\}$ is $2$-connected, and it is neither a complete graph nor a cycle. By \autoref{NONDP}, $g$ must be incident with a $4^{+}$-vertex. 
\end{proof}

Applying \autoref{Reducible} to a special $7$-face, we get the following result. 
\begin{lemma}\label{Special}
If $f$ is a special $7$-face and it controls a $(3, 3, 3, 3)$-face, then each of the face controlled by $(3, 4)$-edge has at least two $4^{+}$-vertices. 
\end{lemma}

Define the initial charge function $\mu(x)$ on $V \cup F$ to be $\mu(v)=d(v)-6$ for $v\in V$ and $\mu(f)=2d(f)-6$ for $f\in F$. By Euler's formula, we have the following equality, 
\[
\sum_{v\in V(G)}(d(v)-6) + \sum_{f\in F(G)}(2d(f)-6) = -12.
\]

We give some discharging rules to change the initial charge function $\mu(x)$ to the final charge function $\mu'(x)$ on $V\cup F$ such that $\mu'(x)\geq0$ for all $x\in V\cup F$, which leads to a contradiction.

The following are the discharging rules.
\begin{enumerate}[label = \bf R\arabic*]
  \item\label{Ru-1} Each 4-face sends $\frac{1}{2}$ to each incident $3$-vertex.
  \item\label{Ru-2} Each $7^+$-face sends $\frac{3}{2}$ to each incident weak $3$-vertex, $\frac{5}{4}$ to each incident semi-weak $3$-vertex, 1 to each incident strong $3$-vertex.
  \item\label{Ru-3} Each $7^+$-face sends 1 to each incident poor $4$-vertex, $\frac{3}{4}$ to each incident semi-rich $4$-vertex, $\frac{1}{2}$ to each incident rich $4$-vertex.
  \item\label{Ru-4} Each $7^+$-face sends $\frac{1}{3}$ to each incident 5-vertex.
  \item\label{Ru-5} Each $(3, 3, 3^{+}, 4^{+})$-face sends $\frac{1}{4}$ to each adjacent poor 7-face and each adjacent special 7-face through $(3, 3)$-edge, respectively.
  \item\label{Ru-6} Each $(3, 4, 3^{+}, 4^{+})$-face and $(3, 4, 4^{+}, 3^{+})$-face sends $\frac{1}{4}$ to each adjacent special 7-face through $(3, 4)$-edge.
\end{enumerate}
It remains to check that the final charge of every element in $V\cup F$ is nonnegative.

(1) Let $v$ be an arbitrary vertex of $G$. 

By \autoref{delta}, $G$ has no $2^{-}$-vertices. If $v$ is a $6^{+}$-vertex, then it is not involved in the discharging procedure, hence $\mu'(v) = \mu(v) =  d(v) - 6 \geq 0$. Next, we may assume that $3 \leq d(v) \leq 5$.

Suppose that $v$ is a $3$-vertex. If $v$ is incident with no $4^{-}$-face, then it is incident with three $7^{+}$-faces, and then $\mu'(v) = \mu(v) + 3 \times 1 = 0$. If $v$ is incident with a 3-face, then the other two incident faces are $9^{+}$-faces by \autoref{NO-ADJ}\ref{2}, and then $\mu'(v) = \mu(v) + 2 \times \frac{3}{2} = 0$. If $v$ is incident with a 4-face, then the other two incident faces are $7^{+}$-faces by \autoref{NO-ADJ}\ref{3}, and then $\mu'(v) = \mu(v) + 2 \times \frac{5}{4} + \frac{1}{2} = 0$.

Suppose that $v$ is a $4$-vertex. By \autoref{NO-ADJ}\ref{3}, $v$ is incident with at most two $4^-$-faces. If $v$ is incident with no $4^{-}$-face, then it is incident with four $7^{+}$-faces, and then $\mu'(v) =  \mu(v) + 4 \times \frac{1}{2} = 0$. If $v$ is incident with exactly one $4^-$-face, then $\mu'(v) = \mu(v) + 2 \times \frac{3}{4} + \frac{1}{2} = 0$. If $v$ is incident with exactly two $4^-$-faces, then $\mu'(v) = \mu(v) + 2 \times 1 = 0$.

Suppose that $v$ is a 5-vertex. By \autoref{NO-ADJ}\ref{3}, $v$ is incident with at most two $4^-$-faces. Therefore, it is incident with at least three $7^{+}$-faces, and $\mu'(v) \geq \mu(v) + 3 \times \frac{1}{3} = 0$. 

(2) Let $f$ be an arbitrary face in $F(G)$. 

Since the distance of triangles is at least two, each $k$-face is adjacent to at most $\lfloor\frac {k}{3}\rfloor$ triangular-faces, thus $f$ contains at most $2\times\lfloor\frac {k}{3}\rfloor$ weak $3$-vertices. As observed above, $G$ has no 5-face and 6-face. If $f$ is a 3-face, then it is not involved in the discharging procedure, and then $\mu'(f) = \mu(f ) = 0$.

Suppose that $f$ is a 4-face. If $f$ is incident with four $3$-vertices, then $\mu'(f) = \mu(f) - 4 \times \frac{1}{2} = 0$. If $f$ is incident with exactly one $4^{+}$-vertex, then $f$ sends at most $\frac{1}{4}$ through each incident $(3, 3)$-edge, and then $\mu'(f) \geq \mu(f) - 3 \times \frac{1}{2} - 2 \times \frac{1}{4} = 0$. If $f$ is incident with at least two $4^{+}$-vertices, then $\mu'(f) \geq \mu(f) - 2 \times \frac{1}{2} - 4 \times \frac{1}{4} = 0$. 

If $f$ is an 8-face, then it is not adjacent to any $3$-face and it sends at most $\frac{5}{4}$ to each incident vertex, and then $\mu'(f) \geq \mu(f) - 8 \times \frac{5}{4} = 0$.

\begin{figure}
\centering
\begin{tikzpicture}
\def\n{9}
\pgfmathsetmacro\i{\n-1}
\foreach \x in {0,...,\i}
{
\def\pointname{A\x}
\coordinate (\pointname) at ($(\x*360/\n - 90:1)$);
}

\filldraw[fill=gray] (A0)
\foreach \x in {0,...,\i}
 {-- (A\x)} -- cycle;
 
\foreach \x in {0,...,\i}
{
\node[circle, inner sep =1, fill, draw] () at (A\x) {};
}

\foreach \y in {0,...,35}
{
\def\pointname{B\y}
\coordinate (\pointname) at ($(\y*90/\n - 90:1.4)$);
}
\draw (A1)--(B6)--(A2);
\draw (A4)--(B18)--(A5);
\draw (A7)--(B30)--(A8);
\draw (A4)--(B15);
\draw (A3)--(B11);
\draw[dotted] (A3)--(B13);
\draw[dotted] (A4)--(B16);
\draw (A6)--(B23);
\draw[dotted] (A6)--(B25);
\draw (A0)--(B35);
\draw[dotted] (A0)--(B1);
\node[circle, inner sep =1, fill, draw] () at (B6) {};
\node[circle, inner sep =1, fill, draw] () at (B18) {};
\node[circle, inner sep =1, fill, draw] () at (B30) {};
\end{tikzpicture}
\caption{A $9$-face incident with exactly five weak $3$-vertices.}
\label{9-face}
\end{figure}
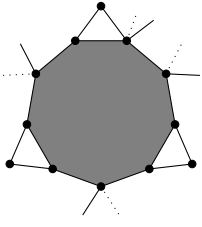

Suppose that $f$ is a $9$-face. Recall that $f$ is incident with at most six weak $3$-vertices. If $f$ is incident with exactly six weak $3$-vertices, then $f$ sends at most $1$ to each other incident vertex, and then $\mu'(f) \geq \mu(f) - 6 \times \frac{3}{2} - (9 - 6) \times 1 = 0$. If $f$ is incident with exactly five weak $3$-vertices, then $f$ must be adjacent to three $3$-faces and one of the six incident vertices on triangles must be a $4^+$-vertex (see \autoref{9-face}), and then $\mu'(f) \geq \mu(f) - 5 \times \frac{3}{2} - 1 - \frac{5}{4} - 2 \times 1 > 0$. If $f$ is incident with exactly four weak $3$-vertices and at least one $4^{+}$-vertex, then $\mu'(f) \geq \mu(f) - 4 \times \frac{3}{2} - 1 - (9 - 4 - 1) \times \frac{5}{4} = 0$. If $f$ is incident with exactly four weak $3$-vertices and no $4^{+}$-vertex, then $f$ is incident with at least one strong $3$-vertex and at most four semi-weak $3$-vertices, and then $\mu'(f) \geq \mu(f) - 4 \times \frac{3}{2} - 4 \times \frac{5}{4} - 1 = 0$. If $f$ is incident with at most three weak $3$-vertices, then $\mu'(f) \geq \mu(f) - 3 \times \frac{3}{2} - (9 - 3) \times \frac{5}{4} = 0$.

If $f$ is a $10^+$-face, then $\mu'(f) \geq \mu(f) - 2 \times \lfloor\frac {d(f)}{3}\rfloor \times \frac{3}{2} - \left(d(f) - 2 \times \lfloor\frac {d(f)}{3}\rfloor\right)\times \frac{5}{4}\geq0$.

Suppose that $f$ is a 7-face. By \autoref{NO-ADJ}\ref{2}, $f$ is not incident with any weak $3$-vertex. It is observed that $f$ is incident with at most six semi-weak $3$-vertices. If there is an incident vertex receives at most $\frac{1}{2}$ from $f$, then $\mu'(f) \geq \mu(f) - \frac{1}{2} - (7 - 1) \times \frac{5}{4} = 0$. So we may assume that $f$ is incident with seven $4^-$-vertices and no rich $4$-vertex. If $f$ is incident with at most four semi-weak $3$-vertices, then $\mu'(f) \geq \mu(f) -  4 \times \frac{5}{4} - 3 \times 1 = 0$. So we may further assume that $f$ is incident with at least five semi-weak $3$-vertices and at most two $4$-vertices. If $f$ is incident with two semi-rich $4$-vertices, then $\mu'(f) = \mu(f) - 2 \times \frac{3}{4} - (7 - 2) \times \frac{5}{4} > 0$. If $f$ is incident with a semi-rich $4$-vertex and a poor $4$-vertex, then $\mu'(f) = \mu(f) - \frac{3}{4} - 1 - (7 - 2) \times \frac{5}{4} = 0$. It is impossible that $f$ is incident with five semi-weak $3$-vertices and two poor $4$-vertices. 

In the following, assume that $f$ is incident with at most one $4$-vertex and at least five semi-weak $3$-vertices. If $f$ is incident with a semi-rich $4$-vertex, then it is incident with at most five semi-weak $3$-vertices, and then $\mu'(f) \geq \mu(f) - \frac{3}{4} - 5 \times \frac{5}{4} - 1 = 0$. Suppose that $f$ is incident with a poor $4$-vertex, then it must be adjacent to six semi-weak $3$-vertices, \ie $f$ is a special 7-face, see \autoref{fig:subfig:-a}. If $f$ controls two $(3, 3, 3^{+}, 4^{+})$-faces through $(3, 3)$-edges, then $\mu'(f) = \mu(f) - 1 - 6 \times \frac{5}{4} + 2 \times \frac{1}{4} = 0$. Then we may assume that $f$ controls at least one $(3, 3, 3, 3)$-face. By \autoref{Special}, $f$ controls two $4$-faces incident with at least two $4^{+}$-vertices through $(3, 4)$-edges, thus $\mu'(f) = \mu(f) - 1 - 6 \times \frac{5}{4} + 2 \times \frac{1}{4} = 0$. Finally, we may assume that $f$ is incident with seven $3$-vertices. In this case, $f$ can only be a poor face, see \autoref{fig:subfig:-b}. By \autoref{Reducible}, $f$ is incident with three $(3, 3, 3^{+}, 4^{+})$-faces, thus $\mu'(f) = \mu(f) - 1 - 6 \times \frac{5}{4} + 3 \times \frac{1}{4} > 0$.
\end{proof}

\subsection{Planar graphs without 4-, 5- and 7-cycles}
The third main result can be derived from the following theorem on degeneracy. 
\MRESULTc*
\begin{proof}
Suppose that $G$ is a planar graph satisfying all the hypothesis but the minimum degree is at least three. Without loss of generality, we may assume that $G$ is connected and it has been embedded in the plane. 
\begin{lemma}\label{NOADJ}
\mbox{}
\begin{enumerate}[label = \bf(\alph*)]
\item\label{NOADJ-1} There is no $4$-, $5$-, $7$-faces. Every 6-face is bounded by a 6-cycle. 
\item\label{NOADJ-2} A 3-face cannot be adjacent to a $7^{-}$-face.
\end{enumerate}
\end{lemma}
\begin{proof}
\ref{NOADJ-1} Since every $4$-face must be bounded by a $4$-cycle, but there is no $4$-cycles in $G$, this implies that there is no $4$-faces in $G$. Similarly, there is no $5$-faces in $G$. Since there is no $7$-cycles in $G$, there is no $7$-face bounded by a cycle, and then the boundary of every $7$-face must consist of a triangle and a $4$-cycle, but this contradicts the absence of $4$-cycles. If the boundary of a $6$-face is not a cycle, then it must consist of two triangles, and the distance of these two triangles is zero, a contradiction. Therefore, every $6$-face is bounded by a $6$-cycle. 

\ref{NOADJ-2} It is easy to check that every $8^{-}$-cycle is chordless. Since there is no two triangles at distance less than two, there is no two adjacent $3$-faces. Every $6$-face is bounded by a $6$-cycle and it is chordless, thus a $3$-face cannot be adjacent to a $6$-face, for otherwise they form a $7$-cycle with a chord, a contradiction. 
\end{proof}

Define the initial charge function $\mu(x)$ on $V \cup F$ to be $\mu(v)=d(v)-6$ for $v\in V$ and $\mu(f)=2d(f)-6$ for $f\in F$. By Euler's formula, we have the following equality, 
\[
\sum_{v\in V(G)}(d(v) - 6) + \sum_{f\in F(G)}(2d(f) - 6) = -12. 
\]

Next, we define some discharging rules to change the initial charge function $\mu(x)$ to the final charge function $\mu'(x)$ on $V\cup F$ such that $\mu'(x)\geq0$ for all $x\in V\cup F$. This leads to a contradiction, and then we complete the proof.

\begin{enumerate}[label = \bf R\arabic*]
  \item\label{Ru-2-} Each $6^+$-face sends $1$ to each incident strong $3$-vertex, $\frac{1}{2}$ to each incident rich $4$-vertex, $\frac{1}{4}$ to each incident $5$-vertex.
  \item\label{Ru-3-} Each $8^+$-face sends $\frac{3}{2}$ to each incident weak $3$-vertex, $\frac{3}{4}$ to each incident semi-rich $4$-vertex.
\end{enumerate}
It remains to check that the final charge of every element in $V\cup F$ is nonnegative.

$\bullet$ Let $v$ be an arbitrary vertex of $G$. 

If $v$ is a $6^{+}$-vertex, then it is not involved in the discharging procedure, hence $\mu'(v) = \mu(v) =  d(v) - 6 \geq 0$. We may assume that $3 \leq d(v) \leq 5$. Since there is no two triangles at distance less than two, every vertex is incident with at most one 3-face. 

Suppose that $v$ is a $3$-vertex. If $v$ is not incident with any $3$-face, then it is incident with three $6^{+}$-faces, and then $\mu'(v) = \mu(v) + 3 \times 1 = 0$. If $v$ is incident with a 3-face, then the other two incident faces are $8^{+}$-faces by \autoref{NOADJ}\ref{NOADJ-2}, and then $\mu'(v) = \mu(v) + 2 \times \frac{3}{2} = 0$. 

Suppose that $v$ is a $4$-vertex. If $v$ is not incident with any $3$-face, then it is incident with four $6^{+}$-faces, and then $\mu'(v) =  \mu(v) + 4 \times \frac{1}{2} = 0$. If $v$ is incident with a $3$-face, then $\mu'(v) = \mu(v) + 2 \times \frac{3}{4} + \frac{1}{2} = 0$. 

Suppose that $v$ is a 5-vertex. Since $v$ is incident with at most one $3$-face, it is incident with at least four $6^{+}$-faces, so $\mu'(v) \geq \mu(v) + 4 \times \frac{1}{4} = 0$. 

$\bullet$ Let $f$ be an arbitrary face in $F(G)$. 

Note that there is no $4$-, $5$-, $7$-faces. Since every 3-face $f$ is not involved in the discharging procedure, we have that $\mu'(f) = \mu(f ) = 0$. By \autoref{NOADJ}\ref{NOADJ-2}, every $6$-face $f$ is adjacent to six $6^{+}$-faces, thus $\mu'(f) \geq \mu(f) - 6 \times 1 = 0$. Suppose that $f$ is a $d$-face with $d \geq 8$. Since the distance of triangles is at least two, we have that $f$ is adjacent to at most $\lfloor\frac {d}{3}\rfloor$ triangular-faces, thus it is incident with at most $2\times\lfloor\frac {d}{3}\rfloor$ weak $3$-vertices. Hence, $\mu'(f) \geq 2d - 6 - 2\times\lfloor\frac {d}{3}\rfloor \times \frac{3}{2} - (d - 2\times\lfloor\frac {d}{3}\rfloor) \times 1 = d - 6 - \lfloor\frac {d}{3}\rfloor \geq 0$. 
\end{proof}

\vskip 0mm \vspace{0.3cm} \noindent{\bf Acknowledgments.} This work was supported by the National Natural Science Foundation of China and partially supported by the Fundamental Research Funds for Universities in Henan (YQPY20140051).

\end{document}